\documentclass[11pt,a4paper,reqno]{amsart}
\usepackage{amsmath}
\usepackage{amsfonts}
\usepackage{amssymb}
\usepackage{amscd}
\usepackage{url}
\usepackage[pdftex,bookmarks=true]{hyperref}

\newcommand\Vol{{\operatorname{Vol}}}

\newcommand\R{{\mathbf{R}}}
\newcommand\C{{\mathbf{C}}}

\renewcommand\P{{\mathbf{P}}}
\newcommand\E{{\mathbf{E}}}

\newcommand\Z{{\mathbf{Z}}}

\newcommand\F{{\mathbf{F}}}

\renewcommand\mod{\ \operatorname{mod}\ }

\newcommand\ep{\epsilon}

\theoremstyle{plain}
  \newtheorem{theorem}[subsection]{Theorem}

  \newtheorem{lemma}[subsection]{Lemma}
  \newtheorem{corollary}[subsection]{Corollary}

\theoremstyle{remark}
  \newtheorem{remark}[subsection]{\bf Remark}
  \newtheorem{properties}[subsection]{\bf Property}

\theoremstyle{definition}
  \newtheorem{definition}[subsection]{Definition}

\include{psfig}

\begin{document}

\title[A new approach to an old problem of Erd\H{o}s and Moser]{A new approach to an old problem of Erd\H{os} and Moser}

\author{Hoi H. Nguyen}
\address{Department of Mathematics, University of Pennsylvania, Philadelphia, PA 19104}
\email{hoing@math.upenn.edu}

\maketitle

\begin{abstract}
Let $\eta_i, i=1,\dots, n$ be iid Bernoulli random variables, taking values $\pm 1$ with probability $\frac{1}{2}$. Given
a multiset $V$ of $n$ elements  $v_1, \dots, v_n$ of an additive group $G$,  we define the \emph{concentration probability} of $V$ as

$$\rho(V) := \sup_{v\in G} \P( \eta _1 v_1 + \dots
\eta_n v_n =v). $$

An old result of Erd\H{o}s and Moser asserts that if $v_i $ are distinct real numbers then $\rho(V)$ is $O(n^{-\frac{3}{2}}\log n)$. This bound was then refined by S\'ark\"ozy and Szemer\'edi to $O(n^{-\frac{3}{2}})$, which is sharp up to a constant factor. The ultimate result dues to Stanley who used tools from algebraic geometry to give a complete description for sets having optimal concentration probability; the result now becomes classic in algebraic combinatorics.

In this paper, we will prove that the optimal sets from Stanley's work are stable. More importantly, our result gives an almost complete description for sets having large concentration probability.

\end{abstract}

\section{Introduction}\label{section:introduction}

Let $\eta_i, i=1,\dots, n$ be iid Bernoulli random variables, taking values $\pm 1$ with probability $\frac{1}{2} $. Given
a multiset $V$ of $n$ elements $v_1, \dots, v_n$ of an additive group $G$, we define the random
walk $S$ with steps in $V$ to be the random variable $S:=\sum_{i=1}^n{\eta_i v_i}$. The \emph{concentration probability} of $V$ is defined to be

$$\rho(V) := \sup_{v \in G} \P(S=v). $$

Motivated by their study of random polynomials, in the 1940s Littlewood and
Offord \cite{LO} raised the question of bounding $\rho(V)$ where $v_i$ are real numbers. They
showed that $\rho(V)=O(n^{-\frac{1}{2}}\log n)$.
Shortly after Littlewood-Offord paper, Erd\H{o}s \cite{ELO}
gave a beautiful combinatorial proof of the
refinement

\begin{equation} \label{eqn:Erdos} \rho(V)\le \frac{ \binom{n}{n/2}}{2^n } =O(n^{-\frac{1}{2}}). \end{equation}

Erd\H os'  result is sharp, as demonstrated by  $V=\{1,\dots,1\}$.

The Littlewood-Offord and Erd\H{o}s results are
classic  in combinatorics and have  generated an impressive wave of research, in
particular from the early 1960s to the late 1980s.

One direction of research was to generalize Erd\H os' result  to
other groups. For example, in  1966 and  1970, Kleitman extended Erd\H{o}s' result to
complex numbers and normed vectors, respectively. 

Another direction was  motivated by the observation that \eqref{eqn:Erdos}  can be improved significantly
under additional assumptions on $V$. The first such result was discovered by  Erd\H{o}s and Moser \cite{EM},
who showed that if $v_i$ are distinct real numbers then $\rho(V) = O(n^{-\frac{3}{2}} \log n)$.
They conjectured that the logarithmic term is not necessary. S\'ark\"ozy and Szemer\'edi claimed this conjecture in the 1960s.

\begin{theorem}[S\'ark\"ozy-Szemer\'edi's theorem, \cite{SS}] \label{theorem:SS} Let $V$ be a set of $n$ distinct integers, then

$$\rho(V)=O(n^{-\frac{3}{2}}).$$ \end{theorem}

Unfortunately, the result of S\'ark\"ozy and Szemer\'edi does not cover the most interesting question (also raised by Erd\H{o}s and Moser \cite{EM}): which sets have optimal concentration probability? This major problem remained open until the breakthrough of Stanley \cite{Stan}  in the early 1980s.

\begin{theorem}[Stanley's theorem] \label{theorem:Stan}
Let $n$ be odd and $V_0 :=\{- \lfloor \frac{n}{2} \rfloor , \dots, \lfloor \frac{n}{2} \rfloor \}$. Let $V$ be any set of $n$ distinct integers, then $$\rho(V) \le \rho(V_0).$$
\end{theorem}

A similar result holds for the case $n$ is even \cite{Stan}.
Stanley's proof of  Theorem \ref{theorem:Stan} used sophisticated machineries  from algebraic geometry, in particular the hard-Lepschetz theorem.
Few years later, a more elementary proof was given by Proctor \cite{P}. This proof, now becomes classic in algebraic combinatorics, is also of algebraic nature, involving the representation of the Lie algebra $sl(2, \C)$.

Although Stanley's result has resolved the question of Erd\H{o}s and Moser, the algebraic approach has some disadvantage. First, it does not yield the actual value of $ \rho(V_0)$. From Theorem \ref{theorem:SS}, one would guess (under the assumption that the elements of $V$ are different)  that $  \rho(V_0) = (C_0+o(1)) n^{-\frac{3}{2}}$ for some constant $C_0 >0$. The algebraic proofs does not confirm this estimate. (Although one can calculate $C_0$ by a basic analytical method.)

Assume that $C_0$ exists for a moment, one would next wonder if $V_0$ is a stable maximizer. In other words, if
some other set $V$ has $\rho(V)$  close to $C_0 n^{-\frac{3}{2}}$, then should $V$, in some sense, close to $V_0 $ ?
(Notice that $\rho$ is invariant under dilation.)

The goal of this paper is to address these issues.


\begin{theorem}[Stability theorem]\label{theorem:stable} There is a positive constant $\epsilon_0$ such that for any $0 < \epsilon \le \epsilon_0$, there is a positive number $\epsilon' = \epsilon'(\epsilon)$ such that $\ep'\rightarrow 0$ as $\ep \rightarrow 0$, and the following holds: if $V$ is a set of $n$ distinct integers and

    $$\rho(V)\ge (\sqrt {\frac{24}{\pi}} -\ep) n^{-\frac{3}{2}},$$

    then there is an integer $l$ which divides all $v\in V$ and $$\sum_{v\in V}(\frac{v}{l}) ^2 \le (1+\ep')\sum_{v\in V_0}v^2.$$
\end{theorem}

It follows rather explicitly from the proof of Theorem \ref{theorem:stable} that

\begin{equation}\label{eqn:maximumvalue}
\rho(V)\le (\sqrt {\frac{24}{\pi}} +o(1)) n^{-\frac{3}{2}},
\end{equation}

for any set $V$ of $n$ distinct elements of a torsion free group. Here we note that

\begin{equation}\label{eqn:V_0}
\rho(V_0)=(\sqrt {\frac{24}{\pi}} +o(1)) n^{-\frac{3}{2}}.
\end{equation}

As a byproduct, we obtain the first non-algebraic proof for the asymptotic version of Stanley theorem, and in a more general way: our elements can be from any torsion free groups.
More importantly, in case $V\subset \Z$, theorem \ref{theorem:stable} and its proof reveal a natural reason for $V_0$ to be  the optimal set:
{\it This is the set (modulo a dilation) that minimizes the variance  $\sum_{v \in V} v^2$ of the random sum $S$.}

Our method is motivated by a recent idea of Tao and Vu: in certain problems on random discrete matrices we need to {\it characterize sets of large concentration probability} (under general assumption on $V$ and $\eta_i$). The interested reader may also see \cite{NV},\cite{TVinverse},\cite{TVsing} for more motivations and ideas lying behind this paper.

Theorem \ref{theorem:stable} is proven in two steps. In the first step (Section \ref{section:step1}) we give a complete characterization of sets having optimal concentration probability up to a constant factor. In the second step (Section \ref{section:step2}) we exploit this strong information to verify the stability part.



\vskip2mm

\section{preliminary setting} The first step is to  embed the problem into the finite field $\F_p$ for some prime $p$. In the case when $v_i$ are integers, we simply take $p$ to be a large prime (for instance $p \ge 2^n (\sum_{i=1}^n |v_i| +1)$ suffices). If $V$ is a subset of a general torsion-free group $G$, one can use Theorem \ref{theorem:Freimaniso} (see Appendix.)

From now on, we can assume that $v_i$ are elements of $\F_p$ for some large prime $p$. We view elements of $\F_p$ as integers between $0$ and $p-1$. We use short hand $\rho$ to denote $\rho (V)$.

{\it Fourier Analysis.} The main advantage of working in $\F_p$ is that one can make use of discrete   Fourier analysis. Assume that $$\rho= \rho(V)=\P( S=a), $$ for some $a  \in \F_p$. Using the standard notation $e_p(x)$ for $\exp(\frac{2\pi \sqrt{-1} x}{p} )$, we have

\begin{equation}\label{eqn:fourier1} \rho= \P(S=a)= \E \frac{1}{p} \sum_{\xi\in \F_p} e_p (\xi (S-a)) = \E \frac{1}{p} \sum_{\xi\in \F_p} e_p (\xi S) e_p(-\xi a) .\end{equation}

By independence

\begin{equation} \label{eqn:fourier2}  \E e_p(\xi S) = \prod_{i=1}^n e_p(\xi \eta_i v_i)= \prod_{i=1}^n \cos \frac{2\pi \xi v_i}{p}  \end{equation}

Since $|e_p(-\xi a)|=1 $, it follows that

\begin{equation} \label{eqn:fourier3} \rho  \le \frac{1}{p} \sum_{\xi \in \F_p} |\cos \frac{2 \pi v_i \xi}{p}  |  = \frac{1}{p} \sum_{\xi \in \F_p} |\frac{\cos  \pi v_i \xi}{p}  | . \end{equation}

 By convexity, we have that   $|\sin  \pi z | \ge 2 \|z\|$ for any $z\in \R$, where $\|z\|:=\|z\|_{\R/\Z}$ is the distance of $z$ to the nearest integer. Thus,

\begin{equation} \label{eqn:fourier3-1}| \cos \frac{\pi x}{p}|  \le  1- \frac{1}{2} \sin^2 \frac{\pi x}{p}  \le 1 -2 \|\frac{x}{p} \|^2  \le \exp( - 2\|  \frac{x}{p} \| ^2 ) ,\end{equation}
where in the last inequality we used that fact that $1-y \le \exp(-y)$ for any $0 \le y \le 1$.

Consequently, we obtain a key inequality

\begin{equation} \label{eqn:fourier4}
\rho \le \frac{1}{p} \sum_{\xi \in \F_p} \prod_{i}|\cos \frac{ \pi v_i \xi}{p}  | \le  \frac{1}{p} \sum_{\xi \in F_p} \exp( - 2\sum_{i=1}^n  \| \frac{v_i \xi}{p}  \| ^2).
\end{equation}

{\it Large level sets.}  Now we consider the level sets $S_m:=\{\xi: \sum_{i=1}^n  \| \frac{v_i \xi}{p} \| ^2 \le m  \} $. Decompose $\F_p$ into $\{0\}\cup S_1\cup (S_2\backslash S_1)\cup\dots$, we have

$$n^{-C} \le \rho  \le  \frac{1}{p} \sum_{\xi \in \F_p} \exp( -2 \sum_{i=1}^n  \| \frac{v_i \xi }{p} \| ^2) \le \frac{1}{p} + \frac{1}{p} \sum_{m \ge 1} \exp(-2(m-1)) |S_m| .$$

Since $\sum_{m\ge 1} \exp(-m) < 1$, there must be  is a large level set $S_m$ such that

\begin{equation} \label{eqn:level1} |S_m| \exp(-m+2) \ge  \rho  p. \end{equation}

We will use (\ref{eqn:level1}) as the starting point for our analysis.

For the rest of this section, we show that $S_{\frac{n}{100}} $ can be viewed as some sort of a {\it dual} set of $V$. In fact, one can show as far as cardinality is concerned, it does behave like a dual

\begin{equation}\label{eqn:dual}
|S_{\frac{n}{100}}|\le \frac{8p}{|V|}=\frac{8p}{n}.
\end{equation}

To see this, define $T_a :=\sum_{\xi \in V} \cos \frac{2\pi a \xi}{p}$. Using the fact that $\cos 2\pi z \ge 1 -50 \|z\|^2 $ for any $z \in \R$, we have, for any $a \in S_{\frac{n}{100}}$

$$T_a \ge  \sum_{\xi \in V} (1- 50 \| \frac{a\xi}{p} \|^2 ) \ge \frac{1}{2} n. $$

One the other hand, using the basic identity

$$
\sum_{a \in \F_p} \cos \frac{2\pi  ax}{p}=
\begin{cases}
p & \mbox{ if } x=0, \\
0 & \mbox{otherwise}.
\end{cases}
$$

We have

$$\sum_{a \in \F_p} T_a^2 \le 2p n . $$

\eqref{eqn:dual} follows from the last two estimates and averaging.

\section{A complete characterization of sets having large concentration probability}\label{section:step1}

The goal of this section is to demonstrate the fact that as long as $\rho(V)=\Omega(n^{-\frac{3}{2}})$ in $\F_p$, where $p$ is large enough, then most of the elements of $V$ belong to a short interval after an appropriate dilation.

\begin{theorem}[Characterization of sets having large concentration probability in $\F_p$]\label{theorem:weakstable}
Let $\delta$ be a positive constant. Then there are constants
$C_1=C_1(\delta)>0$ and $C_2=C_2(\delta)>0$ such that the following
holds. Let $V=\{v_1,\dots,v_n\}$ be a subset of size $n$ of
$\F_p$, where $p \gg n$ is a large prime, such that $\rho(V)\ge
\delta n^{-\frac{3}{2}}$. Then there exists a number $k=k(p)\in \F_p$ and a
partition $V=V_1\cup V_2=k\cdot (W_1\cup W_2)$ with the following properties:

\begin{itemize}

\item $|W_1|\le C_1$,

\vskip .1in

\item $\sum_{w\in W_2} w^2 \le \frac{C_2n^3}{p^2}$.
\end{itemize}

\end{theorem}

The appearance of $k$ is necessary as $\rho (V)$ is invariant under dilation in $\F_p$.

\begin{proof} (of Theorem \ref{theorem:weakstable}) By \eqref{eqn:level1} we have

\begin{align}\label{eqn:3/2:1}
|S_m|\ge \exp(\frac{m}{4}-2) \rho p \ge \delta \exp(\frac{m}{4}-2) pn^{-\frac{3}{2}}
\end{align}

for some $m=O(\log n)$ (it will turn out that $m=O(1)$ later on).

\emph{Structure of $S_m$}. Consider a set sequence, $S_m, 2S_m,\dots, 2^{l}S_m$, where $l$ is the largest integer such that $4^{l}m\le \frac{n}{100}$. Assume that $i_0$ is the smallest index such that $|2^iS_{m}|\ge 2.1 |2^{i-1}S_m|$ for $1\le i\le i_0$ (thus $|2^{i_0}S_m|\ge (2.1)^{i_0} |S_m|.$)

\noindent By the definition of level sets, and by Cauchy-Schwarz inequality we observe that $kS_m\subset S_{k^2m}$ holds for all $k$. In particular, 

$$2^lS_m \subset S_{4^lm} \subset S_{\frac{n}{100}}.$$

On the other hand, by the Cauchy-Davenport theorem (see \cite[Theorem 5.4]{TVbook}), we have 

$$|2^lS_m| \ge 2^{(l-i_0)}(|2^{i_0}S_{m}|-1).$$ 

Hence,

\begin{align*}
|S_{\frac{n}{100}}|&\ge |2^lS_{m}|\ge  2^{(l-i_0)}(|2^{i_0}S_{m}|-1)\\
&\ge 2^{(l-i_0)}(2.1^{i_0}|S_m| -1)\\
&\ge 2^{l-1} (\frac{2.1}{2})^{i_0}|S_m|\\
&\ge 2^{l-1}(\frac{2.1}{2})^{i_0} \delta \exp(\frac{m}{4}-2)p n^{-\frac{3}{2}}.
 \end{align*}

Insert the  estimate $|S_{\frac{n}{100}}| \le \frac{8p}{n}$ from (\ref{eqn:dual}), and $2^l=\Theta({\sqrt{\frac{n}{m}}})$ into the above inequalities, we derive that $m=O_\delta(1)$ and $i_0=O_\delta(1)$.

Now we consider the set $X:=2^{i_0}S_m$. For $|X|\le |S_{\frac{n}{100}}|\le \frac{8p}{n}$, Theorem \ref{theorem:rectifiable} implies that $X$ is
Freiman-isomorphism of order $2h$ to a subset $X'$ of the integers ($h$ is a sufficiently large constant to be chosen).

Next, since $|2X'|=|2X|\le 2.1 |X|=2.1|X'|$, we apply Theorem
\ref{theorem:Freiman'}, and then Lemma \ref{lemma:Sarkozy} (since $h$ is large) to obtain an
arithmetic progression $P'$ of rank 1 and of size $|P'|=
\Theta(|X'|)$ such that $P'\subset hX'$.

Lifting back to $\F_p$, we conclude that $hX$ contains an arithmetic
progression $P$ of rank one and of size $|P|=\Theta(|X|)\gg pn^{-\frac{3}{2}}$ in $\F_p$. Since $X$ is
symmetric, we may assume that $P$ is symmetric. Since $hX = h2^{i_0}S_{m}\subset S_{h^24^{i_0}m}$, we conclude that, with $M:=h^24^{i_0}m =O_\delta(1)$,

$$P\subset hX  \subset S_{M}.$$

We have showed that  there exists a constant $C=O_\delta(1)$ and a symmetric arithmetic progression with rank one and with step $k\in \F_p$, $P=\{ i k :  -Cpn^{-\frac{3}{2}} \le i\le Cpn^{-\frac{3}{2}}\}$ such that the following holds for any $\xi\in P$

$$\sum_{i=1}^n \|\frac{\xi v_i}{p}\|^2 \le M.$$

Set 

$$W:=k^{-1}\cdot V (\mbox{ in } \F_p).$$

We then have

$$\sum_{w\in W} \|\frac{jw}{p}\|^2 \le M$$

for all $j$ from the range $0\le j \le Cpn^{-\frac{3}{2}}.$

Summing over $j$ we get

\begin{equation}\label{eqn:3/2:2}
\sum_{w\in W} \sum_{0\le j\le Cpn^{-\frac{3}{2}}}\|\frac{jw}{p}\|^2 \le CMpn^{-\frac{3}{2}}.
\end{equation}

Let $C'$ be a sufficiently large constant, we set 

$$W_1:=\{w \in W : \|\frac{w}{p}\| \ge  \frac{C' n^{\frac{3}{2}}}{p}\}.$$

Viewing $\F_p$ as $\{0,\dots, p-1\}$, we observe that if $w\in W_1$ then $C'n^{\frac{3}{2}} \le w \le p-C'n^{\frac{3}{2}}$. Thus, provided that  $CC'$ is sufficient large and $0\le j \le Cpn^{-\frac{3}{2}}$, there are at least $Cpn^{-\frac{3}{2}}/3$ indices $j$ with $jw \in [\frac{p}{4},\frac{3p}{4}]$ (in $\F_p$). Note that if this is the case, then $\|\frac{jw}{p}\| \ge \frac{1}{4}$. Thus we have 

$$\sum_{0\le j \le Cpn^{-\frac{3}{2}}}\|\frac{jw}{p}\|^2 \ge (\frac{1}{4})^2Cpn^{-\frac{3}{2}}/3.$$

Summing over $w\in W_1$ we obtain

$$\sum_{w\in W_1}\sum_{0\le j \le Cpn^{-\frac{3}{2}}}\|\frac{jw}{p}\|^2 \ge \frac{C}{48}|W_1|pn^{-\frac{3}{2}}.$$

Together with (\ref{eqn:3/2:2}), this implies that 

$$|W_1| \le 48M.$$

Set

$$W_2:=W\backslash W_1.$$ 

By definition, $\|\frac{jw}{p}\| = j\|\frac{w}{p}\|$ for all $j\le p(2C'n^{\frac{3}{2}})^{-1}$ and for all $w\in W_2$.
Thus

\begin{align*}
\sum_{0\le j\le pn^{-\frac{3}{2}}}\|\frac{jw}{p}\|^2 &\ge \sum_{0\le j\le p(2C'n^{\frac{3}{2}})^{-1}}j^2 \|\frac{w}{p}\|^2\\
&\ge \frac{1}{64{C'}^3} \frac{p^3}{n^{9/2}}\|\frac{w}{p}\|^2.
\end{align*}

Summing over $w\in W_2$ and using  (\ref{eqn:3/2:2}), we obtain

\begin{align*}
\sum_{w\in W_2} \frac{1}{64{C'}^3} \frac{p^3}{n^{9/2}} \|\frac{w}{p}\|^2 &\le \sum_{w\in W_2}\sum_{0\le j\le Cpn^{-\frac{3}{2}}}\|\frac{jw}{p}\|^2\\
&\le CMpn^{-\frac{3}{2}}.
\end{align*}

Hence, 

$$\sum_{w\in W_2} \|\frac{v}{p}\|^2 \le 64C{C'}^3M \frac{n^3}{p^2},$$ 

concluding the proof.

\end{proof}

\begin{remark}
The idea to divide sumsets into dyadic sequences was first utilized by Szemer\'edi and Vu in \cite{SzV}.
\end{remark}

\section{An estimate on $\rho(V_0)$}\label{section:V_0}

To give the reader an overview for the proof of Theorem \ref{theorem:stable}, we show in this section the estimate \eqref{eqn:V_0} for $\rho(V_0)$

$$\rho(V_0)=(\sqrt{ \frac{24}{\pi} } +o(1)) n^{-\frac{3}{2}}.$$

View the elements of $\F_p$ as integers between $-(p-1)/2$ and $(p-1)/2$, by \eqref{eqn:fourier1} we have

\begin{equation}
\P( S=0) = \frac{1}{p} \sum_{\xi \in \F_p} \prod_{i \in V_0}   \cos \frac{2\pi  i \xi}{p}=\frac{1}{p} \sum_{\xi \in \F_p} \prod_{i \in V_0}   \cos \frac{\pi  i \xi}{p}.
\end{equation}

We split this sum into two parts

$$\Sigma_1 := \frac{1}{p} \sum_{\|\frac{\xi}{p}\| \le \frac{\log^2 n}{n^{3/2} }} \prod_{i \in V_0}   \cos \frac{\pi  i \xi}{p} ,$$

$$ \Sigma_2:=  \frac{1}{p} \sum_{\|\frac{\xi}{p}\| >  \frac{\log^2 n}{n^{3/2} }} \prod_{i \in V_0}  \cos \frac{\pi  i \xi}{p}. $$

We are going to show that

\begin{lemma}  \label{lemma:S1}
$$\Sigma_1  =  \frac{1}{p} \sum_{\|\frac{\xi}{p}\| \le \frac{\log^2 n}{n^{3/2}} } \prod_{i \in V_0}   |\cos \frac{\pi  i \xi}{p}| = (\sqrt{ \frac{24}{\pi} } +o(1)) n^{-\frac{3}{2}}.$$
\end{lemma}

\begin{lemma}  \label{lemma:S2}
$$\Sigma_2 \le   \frac{1}{p} \sum_{\|\frac{\xi}{p} \| >  \frac{\log^2  n}{n^{3/2} }} \prod_{i \in V_0}   | \cos \frac{\pi  i \xi}{p}| \le n^{-3} . $$
\end{lemma}

The two lemmas together  imply that $\rho(V_0) = (\sqrt{ \frac{24}{\pi} } +o(1)) n^{-\frac{3}{2}} $.

\begin{proof} (of Lemma \ref{lemma:S1}) The first equality is trivial, as all $\cos$ are positive in this range of $\xi$. Viewing $\xi$ as an integer with absolute value at most $n^{-\frac{3}{2}} p\log^2  n$, we have

$$\cos \frac{\pi  i \xi}{p} = 1 -(\frac{1}{2} +o(1)) \frac{\pi^2 i^2 \xi^2 }{p^2}   =\exp\left(- (1/2+o(1)) \frac{\pi^2 i^2 \xi^2 }{p^2 }\right) . $$

Since $\sum_{i\in V_0} i^2 =(1+o(1)) \frac{n^3}{12} $, it follows that

$$\Sigma_1 =(1+o(1)) \int_{|x| \le \frac{\log^ 2 n}{n^{3/2}} } \exp\left(- (1/2+o(1)) \frac{n^3 \pi^2}{12} x^2 \right). $$

Setting $y= \sqrt {\frac{n^3\pi^2}{12}} x $, changing variables, and using the gaussian  identity $ \frac{1}{\sqrt {2\pi}}  \int^{\infty}_{-\infty}  \exp(-\frac{y^2}{2} ) dy =1$, we have

$$\Sigma_1 = (1+o(1)) (\frac{n^3 \pi^2}{12} )^{-1/2}  \int^{\infty}_{-\infty}  \exp(- \frac{y^2}{2}) dy  =  (\sqrt{ \frac{24}{\pi} } +o(1)) n^{-\frac{3}{2}} , $$
completing the proof. \end{proof}

\begin{proof} (of Lemma \ref{lemma:S2}). To prove Lemma \ref{lemma:S2} we use the upper bound from \eqref{eqn:fourier3-1}. We split the sum into three subsums, according to the magnitude of  $\| \frac{\xi}{p} \|$. We make frequently use of the simple fact that if
$|i| \|\frac{\xi}{p}\| \le 1/2$ then $\|\frac{i \xi}{p} \| = |i| \|\frac{\xi}{p}\|$.

\noindent {\bf Small $\|\frac{\xi}{p}\|$: $\frac{\log^2n}{n^{3/2}} \le \|\frac{\xi}{p}\|\le  \frac{1}{n}$}.
 For all $i$ with $|i| \le \frac{n}{2}$,  we have  $\|\frac{i \xi}{p} \| = |i|\|\frac{\xi}{p} \|$.

\begin{equation}
\sum_{i\in V_0} \|\frac{i \xi}{p} \|^2= \sum_{i \in V_0 } i^2  \|\frac{\xi}{p} \|^2 \ge \frac{n^3}{12} \frac{\log^4n}{n^3} \ge 4 \log n
\end{equation}

Thus, the contribution of this subsum is bounded from above by $n^{-4}$.

\noindent {\bf Medium $\|\frac{\xi}{p}\|$:  $\frac{1}{n} \le \|\frac{\xi}{p} \| \le \frac{1}{4} $.}  By Cauchy-Schwartz, we have

$$\| \frac{i \xi}{p} \|^2 + \| \frac{i' \xi}{p}\|^2 \ge \frac{1}{2} \| \frac{(i-i')\xi}{p} \|^2. $$

For any $\xi$ in this case, let $W(\xi)$ be the set of pairwise disjoint pairs  $(i,i') \in V_0$
that maximizes the sum $M(\xi):= \sum_{(i,i') \in W(\xi)} |i-i'|^2 $ under the constrain
that    $ |i-i' | \| \frac{\xi}{p}\| \le \frac{1}{2} $ for all $(i,i') \in W(\xi)$. It is easy to check that
$M(\xi) \ge n^c$ for some constant $c>0$ for all $\xi $ in this case (For more details see Lemma \ref{lemma:3/2:arrange}.)
From here one can conclude that the contribution of this subsum is at most    $\exp(n^{-\Omega(1) } )= o(n^{-4})$.

\noindent {\bf Large $\|\frac{\xi}{p}\|$: $\frac{1}{4}  \le \|\frac{\xi}{p} \| \le \frac{1}{2} $.}  By Weyl's equidistribution theorem, the number of $i \in V_0$ such that
$\|\frac{i \xi}{p} \| \ge \frac{1}{4}$ is approximately $n/2$. Thus, for any $\xi$ in this category

\begin{equation}
\sum_{i \in V_0} \|\frac{i \xi}{p} \|^2 \ge \frac{n}{8}.
\end{equation}

Hence, the contribution of this subsum is only $\exp(-\Omega (n))$.

\end{proof}

\section{proof of Theorem \ref{theorem:stable}}\label{section:step2}

First, by applying Theorem \ref{theorem:weakstable}, there exists a $k\in \F_p$ and a partition $V=k\cdot(W_1 \cup W_2)$ such that $|W_1|\le C_1$ and $\sum_{w\in W_2} w^2 \le C_2n^3$, where $C_1$ and $C_2$ are positive constants depending on $\ep$. 

Let $C=C(\ep)$ be a large positive constant  and $c=c(\ep)$ be a  small positive constant. By
setting $C,c$ properly and moving at most $\ep n$ elements of $W_2$ to $W_1$, we can assume the following property for $W_2$.

\vskip .2in

\begin{properties}\label{prop:irreducible} $W_2$ is a {\it $c$-irreducible} subset of size at least $(1-\ep)n$ of the interval $[-Cn,Cn]$. In other words, there is no integer $d\ge 2$ which divides all but $c|W_2|$ elements of $W_2$. 
\end{properties}

\vskip .1in

Indeed, after applying Theorem \ref{theorem:weakstable}, because $\sum_{w\in W_2}w^2\le C_2n^3$, all but $\ep n/2$ elements of $W_2$ belong to the interval $[-Cn,Cn]$, provided that $C$ was chosen to be large enough.

To ensure the irreducibility property, we iterate the following process. First, if $W_2$, viewed as a subset of integers, can be written as $W_2 = d_2\cdot W_2' \cup X_2'$ with some integer $d_2\ge 2$ and $|W_2'|\ge |W_2|-c|W_2|$, then we remove these exceptional $c|W_2|$ elements to $W_1$ and pass to consider $W_2'$, which is a subset of $[-Cn/2,Cn/2]$. Next, if $W_2' = d_2'\cdot W_2'' \cup X_2''$ with some integer $d_2'\ge 2$ and $|W_2''|\ge |W_2'|-c|W_2'|$, then we remove these $c|W_2'|$exceptional elements to $W_1$ and switch to consider $W_2''$, which now belongs to $[-Cn/4,Cn/4]$, and keep iterating the process. Observe that whenever reduction occurs, the size of the new set drops by at most $cn$, meanwhile the size of the intervals they belong to shrinks by a factor of $2$. Thus if $c$ was chosen to be smaller than $\ep(4(3+\log_2C))^{-1}$, then the process must terminate after at most $3+\log_2 C$ steps, as the set under consideration would then have size at least $n-C_1-(3+\log_2C)cn \ge n-\ep n/2$, but the interval containing it would be $[-n/8,n/8]$, a contradiction.

Thus, by applying Theorem \ref{theorem:weakstable}, there exists a $k\in \F_p$ and a partition $V=k\cdot(W_1 \cup W_2) (\mod p)$ such that 

\begin{itemize}
\item $|W_1|\le \ep n$,
\vskip .1in
\item $W_2$ satisfies Property \ref{prop:irreducible}.  
\end{itemize}

In the next step, we pass to consider the concentration probability of $W$ by the following identity

$$\P(S_V=kv) = \P(S_W=v) =\frac{1}{p}\sum_{\xi\in\F_p} e_p(-\frac{v\xi}{2}) \prod_{w\in W} \cos{\frac{\pi w\xi}{p}}.$$

We split the sum into two parts,

$$\Sigma_1:= \frac{1}{p}\sum_{\|\frac{\xi}{p}\|\le \frac{\log^2 n}{n^{3/2}}} e_p(-\frac{v\xi}{2}) \prod_{w\in W} \cos{\frac{\pi w\xi}{p}}$$

$$\Sigma_2:= \frac{1}{p}\sum_{\|\frac{\xi}{p}\|>\frac{\log^2n}{n^{3/2}}} e_p(-\frac{v\xi}{2}) \prod_{w\in
W} \cos{\frac{\pi w\xi}{p}}.$$

We are going to exploit the structure of $W_2$ to show that

\begin{lemma}\label{lemma:S2:step2}

$$\Sigma_2 \le \frac{1}{p}\sum_{\|\frac{\xi}{p}\|\ge \frac{\log^2n}{n^{3/2}}} \prod_{w\in W_2} |\cos{\frac{\pi w\xi}{p}}|\le n^{-3}.$$

\end{lemma}

\begin{proof} Making use of \eqref{eqn:fourier3-1}, we split the sum into three subsums, according to
the magnitude of $\|\frac{\xi}{p}\|$.

\noindent {\bf Small $\|\frac{\xi}{p}\|$}: $ \frac{\log^2n}{n^{3/2}} \le \|\frac{\xi}{p}\|\le \frac{1}{2Cn}$. Because $W_2\subset [-Cn,Cn]$, we easily obtain that

$$\sum_{w\in W_2} \|\frac{w \xi}{p}\|^2 = \sum_{w\in W_2} w^2\|\frac{\xi}{p}\|^2 \ge \frac{(1-\ep)^3n^3}{12}\frac{\log^4n}{n^3} \ge 4\log n.$$

Thus the contribution from this part  is bounded from above by $n^{-4}$.

\noindent {\bf Medium $\|\frac{\xi}{p}\|$}: $\frac{1}{2Cn} \le \|\frac{\xi}{p}\| \le \frac{1}{64C}$.
To handle this part, we first  observe the following simple fact.

\begin{lemma}\label{lemma:3/2:arrange}
Let $a\in \F_p$ be arbitrary. Let $\xi\in \F_p$ and
$l>0$ such that $l\|\frac{\xi}{p}\|\le \frac{1}{2}$. Then the following holds for any sequence $0\le i_1<\dots
<i_m\le l$ with $m\ge 4$

\begin{equation}\label{eqn:3/2:arrange}
\sum_{j=1}^m \|\frac{a+i_j\xi}{p}\|^2 \ge \frac{m^3}{48} \|\frac{\xi}{p}\|^2.
\end{equation}
\end{lemma}

The proof of Lemma \ref{lemma:3/2:arrange} is quite simple, we defer it to Appendix \ref{appendix:arrange}. Now we arrange the elements of
$W_2$ as

$$-Cn < w_1<w_2 < \dots < w_{|W_2|} < Cn.$$

Set $l:=\frac{1}{2\|\xi/p\|}$. Thus $32C\le l\le Cn$. Let $i_1$ be the largest index such that $w_{i_1}-w_1 \le l$. We then move on to choose $i_2>i_1$, the largest index such that $w_{i_2}-w_{i_1+1}\le l$, and so on. By so, we create several blocks of elements of $W_2$ with the property that elements of the same block have difference $\le l$.

Since $w_{i_j+1}-w_{i_{j-1}+1}> l$ for all $j$, the number of blocks is less than $\frac{2Cn}{l} +1
$. Next, we call a block {\it short} if it contains no more
than $\frac{l}{8C}$ elements of $W_2$. The total number of elements
of $W_2$ that belong to short blocks is bounded by
$(\frac{2Cn}{l}+1)(\frac{l}{8C}) \le \frac{|W_2|}{2}.$ Hence there are at least $\frac{|W_2|}{2}$
elements that belong to long blocks.

For simplicity, we divide each long block into smaller
blocks of exactly $\lfloor \frac{l}{8C} \rfloor$ elements. The number of such uniform
blocks is then at least $\frac{1}{2}\frac{|W_2|/2}{l/8C} = \frac{2C|W_2|}{l} \ge \frac{Cn}{l}$. We apply \eqref{eqn:3/2:arrange} to each block (with $m=\lfloor l/8C \rfloor$), and then sum over the collection of all blocks using the fact that $l = \frac{1}{2\|\xi/p\|}$,

$$
\sum_{w\in W_2} \|\frac{w \xi}{p}\|^2  \ge \frac{Cn}{l}\frac{m^3}{48} \|\frac{\xi}{p}\|^2 \\
\gg l^2n \|\frac{\xi}{p}\|^2\gg n.$$

Thus the contribution from this part  is bounded from above by $\exp(-\Omega(n))$.

\noindent {\bf Large $\|\frac{\xi}{p}\|$}: $\frac{1}{64C} \le \|\frac{\xi}{p}\|\le \frac{1}{2}$. Let $\delta: = \frac{1}{\log n}$ and

$$W_2':= \{w\in W_2, \|\frac{w \xi}{p}\| \le \delta\}.$$

Assume for the moment that $|W_2'|\le (1-c)|W_2|$ for some positive constant $c=c(\ep)$. Then 

$$\sum_{w\in W_2} \|\frac{w\xi}{p}\|^2 \ge c |W_2|\delta^2 \gg \delta^2 n \gg \frac{n}{\log^2n},$$ 

and so the contribution from this part is bounded from above by $\exp(-\Omega(n\log^{-2} n))$.

Thus, it suffices  to show that $|W_2'|\le (1-c)|W_2|$  for some sufficiently small $c$. Assume otherwise,  we will deduce that there exists a nontrivial $d\in \Z$ that divides all the elements of $W_2'$, which  contradicts the $c$-irreducibility assumption of $W_2$.

To obtain the above contradiction we first use Lemma \ref{lemma:3/2:smallsarkozy} to $W_2$. It is implied that
that $kW_2-kW_2$ contains an arithmetic progression of rank one, $Q=\{id:|i| \le
\gamma n\}$. 

Because $Q\subset [-2kCn,2kCn]$, the step $d$ must be
bounded,

\begin{equation}\label{eqn:3/2:d}
0<d\le \frac{2kC}{\gamma}.
\end{equation}

Let $q$ be an element of $Q$. By definition,  we can write $q$ as a sum and difference of elements of $W_2$,

$$q=w_1+\dots+w_k-w_{k+1} -\dots -w_k.$$ 

Because $\|\frac{w\xi}{p}\| \le \delta$ for all $i$, by the triangle inequality we have  

$$\|\frac{q\xi}{p}\|\le 2k\delta.$$ 

Apply the above estimate for all elements of $Q$, one obtains 

$$\|\frac{id\xi}{p}\|\le 2k\delta \mbox{ for all } |i|\le \gamma n.$$

But as $2k\delta <\frac{1}{4}$, one must have 

$$\|\frac{d\xi}{p}\|\le \frac{2k\delta}{\gamma n}.$$

Next, view $\F_p$ as the interval $[-(p-1)/2,(p-1)/2]$ of $\Z$ and consider $\xi$ as an integer which satisfies $\frac{p}{64C} \le |\xi|\le \frac{p}{2}$. 

By the above bound of $\|\frac{d\xi}{p}\|$, we can write $d\xi = sp+t$, and so 

$$\xi = \frac{sp+t}{d},$$ 

where $s,t$ are integers with $|t|\le \frac{2k\delta p}{\gamma n}$. 

We now collect some useful information for $s$ and $d$. Firstly, because $\xi$ has large absolute value, $|\xi|\ge \frac{p}{64C}$, $s$ cannot be zero.
Secondly, because $|\xi|\le p/2$ and $t$ are small, $d$ does not divide $s$.

Let $w'$ be an arbitrary element of $W_2'$ and consider in $\Z$ the product $w'\xi$,

\begin{equation}\label{eqn:3/2:w}
w'\xi = \frac{w'sp}{d}+\frac{w't}{d} = s'p + \frac{t'p}{d} + \frac{w't}{d},
\end{equation}

\noindent where $w's = s'd+t'$ with $s',t'\in \Z$, and $-\frac{d}{2}\le t'\le \frac{d}{2}.$

Because $|w'|\le Cn$, we have 

$$|\frac{w't}{d}| \le Cn \frac{2k\delta p/\gamma n}{d}  = \frac{2Ck\delta}{\gamma} \frac{p}{d} \le \frac{p}{2d},$$ 

where in the last inequality we used the fact that $\delta$ is small compared to all other quantities.

We next consider two cases according to the value of $t'$. We will eventually run to contradiction in both cases, thus completing the proof of Lemma \ref{lemma:S2:step2}.

\noindent {\bf Case 1 }: $t'\neq 0$.  Because $-\frac{d}{2}\le t' \le \frac{d}{2}$, we have $\frac{p}{d} \le
|\frac{t'p}{d}|\le \frac{p}{2}$, and so 

$$\frac{p}{2d} \le |\frac{t'p}{d} + \frac{wt}{d}| \le \frac{p}{2} + \frac{p}{2d}.$$

Together with \eqref{eqn:3/2:w}, this implies that $\|\frac{w'\xi}{p}\| \ge \frac{1}{2d}$. Hence, by the bound of $d$ from (\ref{eqn:3/2:d}), 

$$\|\frac{w'\xi}{p}\| \ge \frac{\gamma}{4kC} > \delta .$$ 

This inequality violates the definition of $W_2'$, so Case 1 does not hold as long as $\delta$ is sufficiently small.

\noindent {\bf Case 2}: $t'=0$. It then follows that $d$ divides $sw'$
for all $w'\in W_2'$. But recall that $d$ does not divide $s$, so all the element of $W_2'$ are divisible by a nontrivial divisor of $d$. However, this contradicts with the $c$-irreducibility assumption of $W_2$.

\end{proof}

We have shown that the contribution  of $\Sigma_2$ is negligible. Thus, it suffices to justify Theorem \ref{theorem:stable} from the following assumption

\begin{equation}\label{eqn:3/2:Sigma1}
\Sigma_1\ge (1-\ep)\sqrt{\frac{24}{\pi}}n^{-\frac{3}{2}}=(1-\ep) \sqrt{\frac{12}{\pi^2n^3}} \int_{-\infty}^\infty\exp(-\frac{y^2}{2})dy.
\end{equation}

We have

\begin{align*}
\Sigma_1 &= \frac{1}{p}\sum_{\|\frac{\xi}{p}\|\le \frac{\log^2n}{n^{3/2}}} e_p(-\frac{v\xi}{2}) \prod_{w\in W} \cos{\frac{\pi w\xi}{p}}\\
&\le \frac{1}{p}\sum_{\|\frac{\xi}{p}\|\le \frac{\log^2n}{n^{3/2}}} \prod_{w\in W_1} |\cos{\frac{\pi w \xi}{p}}| \prod_{w\in W_2} \cos{\frac{\pi w\xi}{p}}\\
&\le \frac{1}{p}\sum_{\|\frac{\xi}{p}\|\le \frac{\log^2n}{n^{3/2}}} \exp\left(-\sum_{w\in W_1} \|\frac{\xi w}{p}\|^2\right) f_{W_2}(\xi),
\end{align*}

\noindent where

$$f_{W_2}(\xi):=\prod_{w\in W_2} \cos{\frac{\pi w\xi}{p}} =\exp\left(-(1/2+o(1))\sum_{w\in W_2}\frac{\pi^2 w^2 \xi^2}{p^2}\right).$$

Combining this estimate and the lower bound for $\Sigma_1$ we obtain  that

\begin{equation}\label{eqn:3/2:W2}
\sum_{w\in W_2}w^2 \le (1+\ep)\frac{n^3}{12}.
\end{equation}

What remains is to show that a similar sum involving the elements of $W_1$ is negligible. To this end, we first show that the elements of $W_1$  do not have large absolute values (viewing $\F_p$ as the interval  $[-(p-1)/2,(p-1)/2]$ of $\Z$).

Let $C=C(\ep)$ be a number so that

\begin{equation}\label{eqn:C}
\int_{|y|\ge C} \exp(-\frac{y^2}{2}) dy = \ep.
\end{equation}

As $p$ and $n$ are large and $|W_2|\ge (1-\ep)n$, we have

\begin{align*}
\frac{1}{p}\sum_{\frac{C}{n^{3/2}} \le \|\frac{\xi}{p}\|\le \frac{\log^2n}{n^{3/2}}} f_{W_2}(\xi) &\le (\sum_{w\in W_2}\pi^2w^2)^{-1/2}\int_{|y|\ge C}\exp(-\frac{y^2}{2})dy\\
&< (1+4\ep)^{-1}\sqrt{\frac{12}{\pi^2 n^3}} \ep \le \ep \sqrt{\frac{12}{\pi^2n^3}}\int_{-\infty}^{\infty}\exp(-\frac{y^2}{2})dy.
\end{align*}

It thus follows from \eqref{eqn:3/2:Sigma1} that

\begin{equation}\label{eqn:3/2:Sigma1'}
\frac{1}{p}\sum_{\|\frac{\xi}{p}\|\le \frac{C}{n^{3/2}}}\exp\left( -\sum_{w\in W_1} \|\frac{w\xi}{p}\|^2\right) f_{W_2}(\xi)\ge
(1-2\ep)\sqrt{\frac{12}{\pi^2n^3}}\int_{-\infty}^{\infty}\exp(-\frac{y^2}{2})dy.
\end{equation}

Now we are ready to state our estimate on the elements of $W_1$. 

\begin{lemma}\label{lemma:3/2:W1} View $\F_p$ as the interval  $[-(p-1)/2,(p-1)/2]$ of $\Z$. Let $w_0$ be an arbitrary element of $W_1$. Then,

$$|w_0|\le \frac{1}{2C}n^{\frac{3}{2}}.$$
\end{lemma}

\begin{proof}(of Lemma \ref{lemma:3/2:W1})  Assume for contradiction that $|w_0| >  \frac{n^{\frac{3}{2}}}{2C}$. Then the number of $\xi\in[-\frac{Cp}{n^{3/2}},\frac{Cp}{n^{3/2}}]$ satisfying $\frac{1}{8} \le \|\frac{w_0\xi}{p}\|$ is at least $(3/4-o(1)) \frac{2Cp}{n^{3/2}}$. Denote this set by $I$, then for $\xi\in I$ we have

$$\sum_{w\in W}\|\frac{w\xi}{p}\|^2 \ge \|\frac{w_0\xi}{p}\|^2\ge \frac{1}{64}.$$

Notice that the function $f_{W_2}(\xi)$ is decreasing in $|\xi|$ (in the range $-\frac{Cp}{n^{3/2} }\le \xi \le \frac{Cp}{n^{3/2}}$) and $|I|\ge Cp/n^{3/2}$, we have

\begin{align*}
\Sigma_1 &\le \frac{1}{p}\sum_{\xi\in I}\exp(-\|\frac{w\xi}{p}\|^2) f_{W_2}(\xi) + \frac{1}{p}\sum_{\|\frac{\xi}{p}\|\le \frac{C}{n^{3/2}}, \xi\not\in I} \exp(-\|\frac{w\xi}{p}\|^2) f_{W_2}(\xi)\\
&\le \frac{1}{p}\sum_{\xi\in I}\exp(-\frac{1}{64})f_{W_2}(\xi)+ \frac{1}{p}\sum_{\|\frac{\xi}{p}\|\le \frac{C}{n^{3/2}}, \xi\not\in I}f_{W_2}(\xi)\\
&\le \frac{1}{p}\sum_{\|\frac{\xi}{p}\|\le \frac{C}{n^{3/2}}}f_{W_2}(\xi) - (1-\exp(-\frac{1}{64}))\frac{1}{p}\sum_{\frac{C}{2n^{3/2}}\le \|\frac{\xi}{p}\|\le \frac{C}{n^{3/2}}}f_{W_2}(\xi).
\end{align*}

As $p$ and $n$ are large, the above sum is bounded by

\begin{align*}
\Sigma_1 &\le (1+4\ep)\sqrt{\frac{12}{\pi^2 n^3}}\left( \int_{|y|\le C}\exp(-\frac{y^2}{2}) - (1-\exp(-\frac{1}{64}))\int_{C/2 \le |y|\le C} \exp(-\frac{y^2}{2})dy \right)\\
&\le (1+4\ep)\sqrt{\frac{12}{\pi^2 n^3}} \left( \sqrt{2\pi} -(1-\exp(-\frac{1}{64}))\ep^{2/3} \right)
\end{align*}

\noindent where in the last estimate we used the definition \eqref{eqn:C} of $C$ and the fact that 

$$\int_{C/2\le |y|\le C} \exp(-\frac{y^2}{2}) dy \ge \ep^{2/3}.$$

By choosing $\ep$ sufficiently small, the upper bound for $\sigma_1$ above is
smaller than the lower bound in \eqref{eqn:3/2:Sigma1}, a contradiction.
This  concludes  the proof of Lemma \ref{lemma:3/2:W1}.
\end{proof}

To continue, set 

$$\sigma_1 := \sum_{w\in W_1}w^2.$$ 

By Lemma \ref{lemma:3/2:W1}, one has $|w|\le \frac{1}{2C}n^{3/2}$ for every $w\in W_1$. Thus if $\|\frac{\xi}{p}\|  \le Cn^{-3/2}$, then

$$\sum_{w\in W_1}\|\frac{w\xi}{p}\|^2 = \sum_{w\in W_1}|w|^2\|\frac{\xi}{p}\|^2 = \sigma_1 \|\frac{\xi}{p}\|^2.$$

This observation enables us to apply the proof of Lemma \ref{lemma:3/2:W1}, with $w_0^2$ replaced by $\sigma_1$, to obtain a similar bound for $\sigma_1$

\begin{equation}\label{eqn:sigma1}
\sigma_1 =\sum_{w\in W_1} w^2 \le \frac{n^3}{4C^2}.
\end{equation}

Putting \eqref{eqn:3/2:W2} and \eqref{eqn:sigma1} together, we obtain

\begin{align}\label{eqn:final}
\sum_{w\in W} w^2 = \sum_{w\in W_1}w^2 + \sum_{w\in W_2} w^2 &\le (1+\ep+\frac{3}{C^2})\frac{n^3}{12}\nonumber \\
&\le (1+\ep')\frac{n^3}{12},
\end{align}

where $\ep'\rightarrow 0$ as $\ep \rightarrow 0$.

Let us stop to record what we have achieved so far. For any sufficiently large $p$, there exist a nonzero number $k\in \F_p$ and a set $W$ of integers which satisfies \eqref{eqn:final} and such that $V=k\cdot W (\mod p)$.

On the other hand, note that the number of sets $W$ which satisfy \eqref{eqn:final} is bounded by a function of $n$ and $\ep'$. Hence, there exists a common set $W=\{w_1,\dots,w_n\}$ for which $v_i=kw_i(\mod p), v_i\in V$ for infinitely many primes $p$. 

As a consequence, for any pair $(i,j)$ one has $p|k(v_iw_j-v_jw_i)$, or equivalently

$$p|v_iw_j-v_jw_i.$$

Because this is true for infinitely many $p$, one has 

$$v_iw_j=v_jw_i (\mbox{in } \Z).$$ 

As this identity holds for every pair $(i,j)$, and as $W$ is irreducible (because of either \eqref{eqn:final} or Property \ref{prop:irreducible}), there exists a nonzero integer $l$ such that 

$$V=l\cdot W.$$

In other words, $l$ divides every element of $V$ and 

$$\sum_{v\in V} (\frac{v}{l})^2 = \sum_{w\in W}w^2 \le (1+\ep') \frac{n^3}{12},$$

completing the proof of Theorem \ref{theorem:stable}.

{\it Acknowledgment.} The author would like to thank Van Vu for very helpful discussions and remarks. He is grateful to the referee for many of his thoughtful  comments  which led to a substantial improvement in the presentation of the paper.

\appendix

\section{Proof of Lemma \ref{lemma:3/2:arrange}}\label{appendix:arrange}

Without loss of generality, we assume that $m$ is even. For each $j\le m/2$, by Cauchy-Schwarz
inequality and the triangle inequality, we have

\begin{align*}
\|a_{i_j}\xi\|^2+\|a+i_{m-j}\xi\|^2 &\ge \frac{1}{2} \|(a+i_{m-j}\xi)-(a+i_j\xi)\|^2\\ 
&\ge \frac{1}{2}\|(i_{m-i}-i_j)\xi\|^2\\
&\ge \frac{1}{2}(i_{m-j}-i_j)^2\|\xi\|^2.
\end{align*}

Summing over $0\le j\le m/2$, we obtain 

\begin{align*}
\sum_{j=1}^m \|a_{i_j}\xi\|^2 &\ge \frac{1}{2} \sum_{1\le j\le m/2} (i_{m-j}-i_j)^2\|\xi\|^2\\
&\ge \frac{1}{2} \sum_{1\le j\le m/2} j^2\|\xi\|^2\\
&\ge \frac{m^3}{48}\|\xi\|^2,
\end{align*}

where in the second inequality we used the fact that $i_{m-j}-i_j$ strictly decreases
in $\Z$.

\section{Tools from Additive Combinatorics}\label{apendix:AC}

\subsection{Freiman homomorphism}\label{subsection:FH}
We now introduce the concept of a Freiman homomorphism, that
allows us to transfer an additive problem in one group $G$ to another group $G'$ in a way
which is more flexible than the usual algebraic notion of group homomorphism.

\begin{definition}[Freiman homomorphisms] Let $k\ge 1$, and let $X$, $Y$ be additive
sets of groups $G$ and $H$ respectively. A Freiman homomorphism of order
$k$ from $X$ to $Y$ is a map $\phi:X \rightarrow Y$ with the property that

$$x_1 + \dots +x_k = x_1' +\dots + x_k' \Longrightarrow  \phi(x_1)+\dots+ \phi(x_k ) = \phi(x_1') + \dots+ \phi(x_k')$$

\noindent for all $x_1,\dots,x_k; x_1',\dots, x_k'$. If in addition there is an inverse map $\phi^{-1}$ from  $Y$ to $X$
which is a Freiman homomorphism of order $k$, then we say that $\phi$ is a Freiman isomorphism of order $k$, and that $X$ and $Y$ are
Freiman isomorphic of order $k$.
\end{definition}

Clearly Freiman homomorphisms preserve the property of being a progression. We now mention a result that shows torsion-free additive groups are not richer than the integers, for the purposes of understanding sums and differences of finite sets (\cite[Chapter 5]{TVbook}).

\begin{theorem} Let $X$ be a finite subset of a torsion-free additive group $G$. Then for
any integer $k$, there is a Freiman isomorphism $\phi$ : $X\rightarrow \phi(X)$ of order $k$ to some
finite subset $\phi(X)$ of the integers $\Z$. The same is true if we replace $\Z$ by $\F_p$, if $p$ is sufficiently large depending on $X$.
\end{theorem}

An identical proof also yields  a somewhat stronger result below.

\begin{theorem}\label{theorem:Freimaniso}
Let $X$ be a finite subset of a torsion-free additive group $G$. Then for
any integer $k$, there is a map $\phi$ : $X\rightarrow \phi(X)$ to some
finite subset $\phi(X)$ of the integers $\Z$ such that

$$x_1+\dots +x_i = x_1'+\dots + x_j' \Leftrightarrow \phi(x_1)+\dots +\phi(x_i) = \phi(x_1')+\dots \phi(x_j') $$

\noindent for all $i,j\le k$. The same is true if we replace $\Z$ by $\F_p$, if $p$ is sufficiently large depending on $A$.
\end{theorem}

\subsection{Freiman's inverse theorem}\label{subsection:FI}

A subset $Q$ of an abelian group is a \emph{Generalized arithmetic progression (GAP) of rank $r$} if it can be expressed as in the form
$$Q= \{a_0+ x_1a_1 + \dots +x_r a_r| M_i \le x_i \le M_i' \hbox{ for all } 1 \leq i \leq r\}$$
for some $a_0,\ldots,a_r$ and $M_1,\dots, M_r,M_1',\dots,M_r'$.

It is convenient to think of $Q$ as the image of an integer box $B:= \{(x_1, \dots, x_r) \in \Z^d| M_i \le m_i
\le M_i' \} $ under the linear map $$\Phi: (x_1,\dots, x_d) \mapsto a_0+ x_1a_1 + \dots + x_r a_r. $$
The numbers $a_i$ are the \emph{generators } of $P$, the numbers $M_i',M_i$ are the \emph{dimensions} of $P$, and $\Vol(Q) := |B|$ is the \emph{volume} of $B$. We say that $Q$ is \emph{proper} if this map is one to one, or equivalently if $|Q| = \Vol(Q)$.  For non-proper GAPs, we of course have $|Q| < \Vol(Q)$.
If $-M_i=M_i'$ for all $i\ge 1$ and $a_0=0$, we say that $Q$ is {\it symmetric}.

It is easy to verify that if $X$ is a dense subset of a GAP of bounded rank, then $|X+X|$ has size not much bigger than $X$. The celebrated Freiman's inverse theorem says the converse. This theorem comes in a number of variants; we give one of them below.

\begin{theorem}[Freiman's inverse theorem] \cite[Theorem 5.35]{TVbook}\label{theorem:Freiman'}
Let $\gamma\ge 2$ be a given positive number. Let $X$ be a subset of
a torsion-free group such that $|X+X| \le \gamma |X|$. Then for any
$0< \ep$ there exists a proper symmetric arithmetic progression $Q$
of rank at most $\lfloor \log_2\gamma +\ep \rfloor$ and size
$|Q|=\Theta_{\ep,\gamma}(|X|)$ such that $|X\cap
Q|=\Theta_{\ep,\gamma}(|X|)$.
\end{theorem}

In our setting, after passing to $\F_p$ by a Freiman-isomorphism, in order to apply Theorem \ref{theorem:Freiman'} we need a to pass back to torsion-free groups. The following result by Green and Ruzsa allows us to do so.

\begin{theorem}\cite[Theorem 1.3]{GR}\label{theorem:rectifiable}
Let $X\subset \F_p$ be a set with $|2X| \le \gamma|X|$ and
$|X|=o_{k,\gamma}(p)$. Then X is Freiman isomorphism of order $k$ to
a subset of a torsion-free group.
\end{theorem}

\subsection{S\'ark\"ozy-type theorem in GAP} Assume that $X$ is a dense subset of a GAP $Q$, then the iterated sumsets  $kX$ contains a structure similar to $Q$ (see \cite[Lemma 4.4, Lemma 5.5]{SzV}, \cite[Lemma B3]{T-sol}).

\begin{lemma}[S\'ark\"ozy-type theorem in progressions]\label{lemma:Sarkozy} Let
$Q$ be a proper GAP in a torsion-free group of rank $r$. Let $X\subset Q$ be a subset such that
$|X|\ge \delta |Q|$ for some $0 < \delta < 1$. Then there exists a positive integer $1\le k\ll_{\delta,r} 1$
such that $kX$ contains a GAP $Q'$ of rank $r$ and size $\Theta_{\delta,r}(|Q|)$. Furthermore, the generators of $Q'$ are bounded multiples of the generators of $Q$. If $Q$ and $X$ are symmetric, then $Q'$ can be chosen to be symmetric.
\end{lemma}

An easy corollary when $Q$ has rank 1 follows.

\begin{corollary}\label{lemma:3/2:smallsarkozy}
Assume that $X$ is a subset of $[-Cn,Cn]$ of size $n$ in $\Z$. Then
there exists an integer $k=k(C)$ and a positive number $\gamma
=\gamma(C)>0$ such that $kX-kX$ contains a symmetric arithmetic
progression of rank 1 and length $2\gamma n+1$.
\end{corollary}

\end{document}